\documentclass{amsart}
\title[Cohomologies of Sasakian groups and Sasakian solvmanifolds]
{Cohomologies of Sasakian groups and Sasakian solvmanifolds}

\author{Hisashi Kasuya}

\usepackage{amssymb}
\usepackage{amsmath}
\usepackage{amscd}
\usepackage{amstext}
\usepackage{amsfonts}
\usepackage[all]{xy}

\theoremstyle{plain}

\theoremstyle{plain}

\theoremstyle{plain}

\theoremstyle{plain}
\newtheorem{theorem}{Theorem}[section] 
\theoremstyle{remark}
\newtheorem{remark}{Remark}
\theoremstyle{Main result}
\newtheorem{main result}{Main result}
\theoremstyle{lemma}
\newtheorem{lemma}[theorem]{Lemma}
\theoremstyle{definition}

\theoremstyle{proposition}

\theoremstyle{corollary}
\newtheorem{corollary}[theorem]{Corollary}
\theoremstyle{remark}

\theoremstyle{remark}

\theoremstyle{remark}

\theoremstyle{assumption}

\address[Hisashi Kasuya]{Department of Mathematics, Tokyo Institute of Technology, 2-12-1 Ookayama, Meguro-ku, Tokyo 152-8551, JAPAN}
\email{kasuya@math.titech.ac.jp}


\newcommand{\C}{\mathbb{C}}
\newcommand{\R}{\mathbb{R}}

\newcommand{\Z}{\mathbb{Z}}
\newcommand{\g}{\frak{g}}

\begin{document} 

\maketitle
\begin{abstract}
We show certain symmetry  of the dimensions of  cohomologies  of the fundamental groups of compact Sasakian manifolds by using the Hodge theory of twisted basic cohomology.
As  applications, we show that the polycyclic fundamental groups of compact Sasakian manifolds are virtually nilpotent and Sasakian solvmanifolds are finite quotients of  Heisenberg nilmanifolds.
\end{abstract}
\section{Introduction}

The purpose of this paper is to study the fundamental groups of compact Sasakian manifolds.
Sasakian manifolds constitute an odd-dimensional counterpart of the class of K\"ahler manifolds.
A Riemannian manifold $(M,g)$ is a Sasakian manifold if the cone metric on the cone manifold $M\times \R_{+}$ is a K\"ahler metric. 
The fundamental groups of compact  K\"ahler manifolds satisfy various properties.
Since the first cohomology of  manifolds and the first group cohomology of their fundamental groups are isomorphic,
the first group cohomology of the fundamental group of   a compact K\"ahler manifold was studied precisely by using the Hodge theory.
In this paper, we study the first cohomologies of the fundamental groups of compact Sasakian manifolds for $1$-dimensional representations by using the Hodge theory of the basic cohomology.

Let $\Gamma$ be a group.
We denote by ${\mathcal C}(\Gamma)$ the space of characters $\Gamma\to \mathrm{GL}_{1}(\C)$ which can be factored as
\[\Gamma\to H_{1}(\Gamma,\Z)/({\rm torsion})\to \mathrm{GL}_{1}(\C).
\]
For $\rho\in {\mathcal C}(\Gamma)$, we denote by $H^{\ast}(\Gamma,\C_{\rho})$ the group cohomology of $\Gamma$ with values in the module associated with the $1$-dimensional representation $\rho$.
Considering the exponential map $\C\to  \C^{\ast}=\mathrm{GL}_{1}(\C)$,
we have the surjective map ${\mathcal E}:H^{1}(\Gamma,\C)\to {\mathcal C}(\Gamma)$.
We define the "real" action of $\R^{\ast}$ on ${\mathcal C}(\Gamma)$ such that for  $f_{1}+\sqrt{-1}f_{2}\in H^{1}(\Gamma,\C)$  with $f_{1}, f_{2}\in H^{1}(M,\R)$, the action is given by
\[t\cdot {\mathcal E}(f_{1}+\sqrt{-1}f_{2})={\mathcal E}(tf_{1}+\sqrt{-1}f_{2})
\]
for $t\in\R^{\ast}$.

In this paper we prove the "$\R^{\ast}$-symmetry" of cohomologies of the fundamental group of a compact Sasakian manifold. 

\begin{theorem}\label{locii}
Let $M$ be a compact Sasakian manifold.
Then for each $\rho\in {\mathcal C}(\pi_{1}(M))$ and $t\in \R^{\ast}$, we have
\[\dim H^{\ast}(\pi_{1}(M), \C_{\rho}) =\dim H^{\ast}(\pi_{1}(M), \C_{t\cdot\rho}) .
\]
\end{theorem}
For a group $\Gamma$, we consider the set
\[{\mathcal J}_{k}(\Gamma)=\{\rho\in {\mathcal C}(\Gamma)\vert \dim H^{1}(\Gamma,\C_{\rho})\ge k\}.
\]
For ${\mathcal C}(\pi_{1}(M))\ni \rho ={\mathcal E}(f_{1}+\sqrt{-1}f_{2})$ with $f_{1}, f_{2}\in H^{1}(M,\R)$, ${\mathcal E}(f_{1}+\sqrt{-1}f_{2})$ is fixed by the $\R^{\ast}$-action if and only if $f_{1}=0$ (equivalently $\rho$ is unitary).
Hence we have the following corollary.
\begin{corollary}\label{FIIIn}
Let $M$ be a compact Sasakian manifold.
If there exists a non-unitary character $\rho:\pi_{1}(M)\to \mathrm{GL}_{1}(\C)$ satisfying $H^{1}(\pi_{1}(M),\C_{\rho})\ge k$, then the set ${\mathcal J}_{k}(\pi_{1}(M))$ is an  infinite set.

\end{corollary}

A group $\Gamma$ is polycyclic if it admits a sequence 
\[\Gamma=\Gamma_{0}\supset \Gamma_{1}\supset \cdot \cdot \cdot \supset \Gamma_{k}=\{ e \}\]
of subgroups such that each $\Gamma_{i}$ is normal in $\Gamma_{i-1}$ and $\Gamma_{i-1}/\Gamma_{i}$ is cyclic.
By  Corollary \ref{FIIIn}, we prove the following result which is  analogous to the Arapura-Nori's result in \cite{AN}.

\begin{corollary}\label{VNIL}
Let $M$ be a compact Sasakian manifold.
Suppose that the fundamental group $\pi_{1}(M)$  is polycyclic.
Then $\pi_{1}(M)$ is virtually nilpotent i.e., it admits a nilpotent subgroup of finite index.
\end{corollary}

\begin{remark}
In \cite{Che}, Chen  proves that the solvable fundamental group of a compact Sasakian manifold is virtually nilpotent.
Chen's result is based on Campana's results on K\"ahler orbifolds   \cite{Cam}.
Campana uses algebraic and complex analytic geometrical techniques.
In this paper, we only use standard differential geometrical techniques on transversely K\"ahler foliations.
This gives a technical advantage.
In Chen's proof, we need the existence of quasi-regular Sasakian structure on a Sasakian manifold.
On the other hand, in this paper, we do not use quasi-regularity.

\end{remark}

We consider nilmanifolds and solvmanifolds.
Solvmanifolds (resp. nilmanifolds) are compact homogeneous spaces of solvable (resp. nilpotent) Lie groups.
It is known that every nilmanifold can be represented by  $G/\Gamma$ such that $G$ is a simply connected nilpotent Lie group and $\Gamma$ is a lattice in  $G$  (see \cite{R}).

In \cite{Nil}, it is proved that 
a compact $2n+1$-dimensional nilmanifold admits a Sasakian structure if and only if it is a Heisenberg nilmanifold $H_{2n+1}/\Gamma$
where $H_{2n+1}$ is the  $(2n + 1)$-dimensional Heisenberg Lie group and $\Gamma$ is its lattice.

By  Corollary \ref{VNIL}, we can easily extend the result in \cite{Nil} for solvmanifolds as in \cite{Hn}.
\begin{corollary}
A compact $2n+1$-dimensional solvmanifold admitting a Sasakian structure is a finite quotient of  Heisenberg nilmanifold.
\end{corollary}
\begin{proof}
It is known that the fundamental group of a compact solvmanifold is a torsion-free polycyclic group and solvmanifolds with isomorphic fundamental groups are diffeomorphic (see \cite{R}).
Hence, by Corollary \ref{VNIL},  we can easily show that a compact $2n+1$-dimensional solvmanifold admitting a Sasakian structure is a finite quotient of a Sasakian nilmanifold.
Thus the Corollary follows from the result in \cite{Nil}.

\end{proof}

In particular, we have the following result.

\begin{corollary}
Let $G$ be a  $2n+1$-dimensional  simply connected  solvable Lie group with a lattice $\Gamma$.
We assume that $G$ is completely solvable (i.e.,  for any $g\in G$, all eigenvalues of the adjoint operator ${\rm Ad}_{g}$ are real).
Then the compact  solvmanifold $G/\Gamma$ admits a Sasakian structure if and only if it is a Heisenberg nilmanifold.
\end{corollary}
\begin{proof}
By the Saito's rigidity theorem in \cite{Sai}, if a simply connected completely solvable Lie group  contains a nilpotent lattice, then it is nilpotent.
Hence, if $G/\Gamma$ admits a Sasakian structure, then by Corollary \ref{VNIL} we can easily show that $G$ is nilpotent.
Thus  the Corollary follows from the result in \cite{Nil}.

\end{proof}

{\bf  Acknowledgements.} 

The author  would  like to thank Xiaoyang Chen   for  introducing to his results.
This research is supported by  JSPS Research Fellowships for Young Scientists.

\section{Preliminary}
Let $M$ be a compact  smooth manifold and $A^{\ast}(M)$ the de Rham complex of $M$.
For a $\C$-valued closed $1$-form $\phi\in A^{r}(M)\otimes\C$, we consider the operator $\phi\wedge :A^{r}(M)\otimes\C\to A^{r+1}(M)\otimes\C$ of left-multiplication.
Define $d_{\phi}=d+\phi\wedge$.
Then we have $d_{\phi}d_{\phi}=0$ and hence $(A^{\ast}(M)\otimes\C,d_{\phi})$ is a cochain complex.
We denote by $H^{\ast}(M,\phi)$ the cohomology of  $(A^{\ast}(M)\otimes\C,d_{\phi})$.
The cochain complex $(A^{\ast}(M)\otimes\C,d_{\phi})$ is considered as the de Rham complex with values in the topologically trivial flat bundle $M\times \C$ with the  connection form $\phi$.
Hence the structure of the cochain complex $(A^{\ast}(M)\otimes\C,d_{\phi})$ is determined by the character $\rho_{\phi} : \pi_{1}(M)\to \mathrm{GL}_{1}(\C)$ given by $\rho_\phi(\gamma)=\exp\left({\int_{\gamma}\phi}\right)$.
 We have an isomorphism $H^{1}(M,\phi)\cong H^{\ast}(\pi_{1}(M), \C_{\rho_{\phi}})$.
The map $H^{1}(M,\C)\ni [\phi]\mapsto \rho_{\phi}\in {\mathcal C}(\pi_{1}(M))$ is identified with the map $\mathcal E$ as in Introduction.
Hence the  action of $\R^{\ast}$ on ${\mathcal C}(\pi_{1}(M))$ as in Introduction is given by
\[t\cdot \rho_{\phi}=\rho_{t{\rm  Re}\phi+\sqrt{-1}{\rm Im} \phi}
\]
for $t\in\R^{\ast}$ and $[\phi]\in H^{1}(M,\C)$.

\section{Proof of Theorem \ref{locii}}
\subsection{Basic cohomology}
Let $M$ be a compact $(2n+1)$-dimensional Sasakian manifold with a Sasakian metric $g$ and 
 $\eta$ the contact structure associated with the Sasakian structure.
Take $\xi$ the Reeb vector field.
Let $A^{\ast}(M)$ be the de Rham complex of $M$.
A differential form $\alpha\in A^{\ast}(M)$ is basic if  $\iota_{\xi}\alpha=0$ and $\iota_{\xi}d\alpha=0$.
Denote by $A_{B}^{\ast}(M)$ the differential graded algebra  of the basic differential forms on $M$ and  denote by $H^{\ast}_{B}(M,\R)$ (resp. $H^{\ast}_{B}(M,\C)$) the cohomology of $A_{B}^{\ast}(M)$ (resp. $A_{B}^{\ast}(M)\otimes\C$).
Then it is known that  the inclusion $A_{B}^{\ast}(M)\subset A^{\ast}(M)$ induces a cohomology isomorphism $H^{0}_{B}(M,\R)\cong H^{0}(M,\R)$ and $H^{1}_{B}(M,\R)\cong H^{1}(M,\R)$ (see \cite{BoG}).

Consider the Hodge star operator $\ast:A^{r}(M)\to  A^{2n+1-r}(M)$ for the Sasakian metric $g$.
We define the transverse Hodge star operator $\ast_{T}:A_{B}^{r}(M)\to  A_{B}^{2n-r}(M)$ as $\ast_{T}(\alpha)=\ast(\eta\wedge \alpha)$.
Then we have $\ast_{T}^{2}\alpha=(-1)^{r}\alpha$.
Restricting the  scalar product $\langle ,\rangle:  A^{\ast}(M)\times A^{\ast}(M)\to \R$ on the basic forms  $A_{B}^{\ast}(M)$ we consider formal adjoint $\delta_{B}:A^{r}_{B}(M)\to A^{r-1}_{B}(M)$ of the differential $d$ on $A_{B}^{\ast}(M)$.
Then we have $\delta_{B}=-\ast_{T} d\ast_{T}$.
Consider the basic Laplacian $\Delta_{B}=d\delta_{B}+\delta_{B}d$.
A basic form $\alpha\in A^{r}_{B}(M)$ is harmonic if $\Delta_{B}\alpha=0$.
Denote ${\mathcal H}^{r}_{B}(M)={\rm ker}\, \Delta_{B_{\vert  A_{B}^{r}(M)}}$.
Then we have the Hodge decomposition (see \cite{KT}, \cite{EKA} and \cite{Ti})
\[A^{r}_{B}(M)={\mathcal H}^{r}_{B}(M)\oplus {\rm im} \, d_{\vert  A_{B}^{r-1}(M)}\oplus  {\rm im} \, \delta_{B_{\vert  A_{B}^{r+1}(M)}}.
\]

We have the transverse complex structure on ${\rm ker}\, \iota_{\xi}\subset \bigwedge TM^{\ast}$ and we obtain the bi-grading $A^{r}_{B}(M)\otimes \C=\bigoplus_{p+q=r} A^{p,q}_{B}(M)$ with the bi-differential $d=\partial_{B}+\bar\partial_{B}$.
We consider the formal adjoint $\partial^{\ast}_{B}:A^{p,q}_{B}(M)\to A^{p-1,q}_{B}(M)$ and $\bar\partial_{B}^{\ast}:A^{p,q}_{B}(M)\to A^{p,q-1}_{B}(M)$ of $\partial_{B}$ and $\bar\partial_{B}$ respectively for the restricted Hermitian  inner product  on $A^{\ast}_{B}(M)\otimes \C$.
Then we have $\partial^{\ast}_{B}=-\ast_{T}\bar\partial_{B}\ast_{T}$ and  $\bar\partial^{\ast}_{B}=-\ast_{T}\partial_{B}\ast_{T}$.
Denote $\Delta^{\prime}=\partial_{B}\partial^{\ast}_{B}+\partial^{\ast}_{B}\partial_{B}$ and $\Delta^{\prime\prime}=\bar\partial_{B}\bar\partial^{\ast}_{B}+\bar\partial^{\ast}_{B}\bar\partial_{B}$.

Let $\omega=d\eta$.
Then $\omega$ gives a transverse K\"ahler structure.
We can take a complex coordinate system $(z_{1},\dots, z_{n})$ which is transverse to  $\xi$ such that $\omega$ is a Kahler form on $(z_{1},\dots, z_{n})$.
Define the operator $L:A_{B}^{p,q}(M)\to A^{p+1,q+1}_{B}(M)$ by $L\alpha=\omega\wedge \alpha$ and consider  the formal adjoint $\Lambda:A^{p,q}_{B}(M)\to A^{p-1,q-1}_{B}(M)$ of $L$.
We have $\Lambda=-\ast_{T}L\ast_{T}$.
By using the transverse K\"ahler geometry, we obtain the following K\"ahler identity.
\begin{lemma}\label{id1}
\[\Lambda\partial_{B}-\partial_{B}\Lambda=\sqrt{-1}\bar\partial^{\ast}_{B}\quad \text{ and } \quad \Lambda\bar\partial_{B}-\bar\partial_{B}\Lambda=-\sqrt{-1}\partial^{\ast}_{B}.
\]
\end{lemma}
This implies $\Delta_{B}=2\Delta^{\prime}_{B}=2\Delta_{B}^{\prime\prime}$ and hence we have the Hodge structure
\[{\mathcal H}^{r}_{B}(M)\otimes\C=\bigoplus_{p+q=r}{\mathcal H}^{p,q}_{B}(M) \quad \text{ and } \quad \overline{{\mathcal H}^{p,q}_{B}(M) }={\mathcal H}^{q,p}_{B}(M) 
\]
where ${\mathcal H}^{p,q}_{B}(M)={\rm ker}\, \Delta^{\prime}_{B_{\vert A^{p,q}_{B}(M)}}={\rm ker}\, \Delta^{\prime\prime}_{B_{\vert A^{p,q}_{B}(M)}}$.

\subsection{Twisted basic cohomology}
Let $\phi\in A_{B}^{1}(M)\otimes \C$ be a closed basic $1$-form.
Then $(A_{B}^{\ast}(M)\otimes \C,d_{\phi})$ is a cochain complex.
Denote by $H^{\ast}_{B}(M,\phi)$ the cohomology of this complex.
It is known that there exists a sub-torus $\mathcal T\subset {\rm Isom}(M,g)$ such that $A^{\ast}_{B}(M)\otimes \C\subset (A^{\ast}(M)\otimes \C)^{\mathcal T}$ and we have the exact sequence of complexes 
\[\xymatrix{
0\ar[r]&A^{\ast}_{B}(M)\otimes \C\ar[r]&(A^{\ast}(M)\otimes \C)^{\mathcal T}\ar[r]&A^{\ast-1}_{B}(M)\otimes \C\ar[r]&0
}
\]
for the usual differential $d$ (see \cite[Section 7.2.1]{BoG}).
We can say that this is also exact for twisted differential  $d_{\phi}$.
Hence, taking the long exact sequence, we have the  exact sequence
\[\xymatrix{
0\ar[r]&H^{1}_{B}(M,\phi)\ar[r]&H^{1}(M,\phi)\ar[r]&H^{0}_{B}(M,\phi).
}
\]
We can easily check 
\[H^{0}_{B}(M,\phi)=H^{0}(M,\phi)=0
\]
and hence we have:
\begin{lemma}\label{tw1}
\[H^{1}_{B}(M,\phi)\cong H^{1}(M,\phi).
\]
\end{lemma}
Consider the formal adjoint $(\phi\wedge)^{\ast}_{B}:A^{r}_{B}(M)\otimes \C\to A^{r-1}_{B}(M)\otimes \C$ of the operator $\phi\wedge$  for the restricted Hermitian  inner product  on $A^{\ast}_{B}(M)\otimes \C$.
Then we have 
\[(\phi\wedge)^{\ast}_{B}=\ast_{T}(\bar\phi\wedge)\ast_{T}
\]
(see the proof of \cite[Corollary 2.3]{PR}).
Taking the formal adjoint $(\phi\wedge)^{\ast}:A^{r}(M)\otimes \C\to A^{r-1}(M)\otimes \C$ on the usual de Rham complex $A^{\ast}(M)\otimes\C$, we have
\[(\phi\wedge)^{\ast}_{\vert A^{\ast}(M)\otimes \C}=(\phi\wedge)^{\ast}_{B}.
\]
Let $\delta_{B,\phi}=\delta_{B}+(\phi\wedge)^{\ast}$, $\Delta_{B,\phi}=d_{\phi}\delta_{B,\phi}+\delta_{B,\phi}d_{\phi}$ and ${\rm ker}\, \Delta_{B,\phi_{\vert A^{r}(M)\otimes \C}}={\mathcal H}^{r}_{B}(M,\phi)$.
As in \cite{KT}, \cite{EKA} and \cite{Ti}, we have the Hodge decomposition
\[ A^{r}_{B}(M)\otimes \C={\mathcal H}^{r}_{B}(M,\phi)\oplus {\rm im}\, d_{\phi_{\vert A_{B}^{r}(M)\otimes \C}}\oplus  {\rm im}\, \delta_{\phi_{\vert A_{B}^{r}(M)\otimes \C}}.
\]

Consider the double complex $(A^{\ast,\ast}_{B}(M),\partial_{B},\bar\partial_{B})$.
For a $(1,0)$-basic form $\theta$, considering the operators $\theta\wedge:A^{p,q}_{B}(M)\to A^{p+1,q}_{B}(M)$ and $\Lambda:A^{p,q}_{B}(M)\to A^{p-1,q-1}_{B}(M)$,
as the local argument for the K\"ahler identities, see, e.g., \cite[Lemma 6.6]{Voi}, we have the following identity.
\begin{lemma}\label{id2}
\[\Lambda(\theta\wedge) -(\theta\wedge) \Lambda=-\sqrt{-1}\ast_{T}(\theta\wedge)\ast_{T}\quad \text{ and } \quad \Lambda(\bar\theta\wedge) -(\bar\theta\wedge) \Lambda=\sqrt{-1}\ast_{T}(\bar\theta\wedge)\ast_{T}.
\]
\end{lemma}
Let $\phi\in {\mathcal H}^{1}_{B}(M)\otimes\C$.
By ${\mathcal H}^{1}_{B}(M)={\mathcal H}^{1,0}_{B}(M)\oplus {\mathcal H}^{0,1}_{B}(M)$, we can take unique $\theta_{1},\theta_{2}\in {\mathcal H}^{1,0}_{B}(M)$ such that
$\phi=\theta_{1}+\bar\theta_{1}+\theta_{2}-\bar\theta_{2}$.
Define
\[\partial_{B,\theta_{1},\theta_{2}}=\partial_{B}+\theta_{2}\wedge+\bar\theta_{1}\wedge \quad \text{ and } \quad \bar\partial_{B,\theta_{1},\theta_{2}}=\bar\partial_{B}-\bar\theta_{2}\wedge +\theta_{1}\wedge.
\]
By ${\mathcal H}^{p,q}_{B}(M)={\rm ker}\, \Delta^{\prime}_{B_{\vert A^{p,q}_{B}(M)}}={\rm ker}\, \Delta^{\prime\prime}_{B_{\vert A^{p,q}_{B}(M)}}$,
$(A_{B}^{\ast}(M)\otimes \C, \partial_{B,\theta_{1},\theta_{2}})$ and $(A_{B}^{\ast}(M)\otimes \C, \bar\partial_{B,\theta_{1},\theta_{2}})$ are cochain complexes.
Denote by $H^{\ast}_{B}(M,\theta_{1},\theta_{2})$ the cohomology of $(A_{B}^{\ast}(M)\otimes \C, \bar\partial_{B,\theta_{1},\theta_{2}})$.
As similar to \cite[Lemma 2.1]{KHH}, we obtain the following lemma.
\begin{lemma}\label{RRA}
For any $t\in \R^{\ast}$, we have
\[\dim H^{\ast}_{B}(M,\theta_{1},\theta_{2})=\dim H^{\ast}_{B}(M,t\theta_{1},\theta_{2}).
\]
\end{lemma}
Consider the formal adjoint 
\[\partial_{B,\theta_{1},\theta_{2}}^{\ast}=\partial_{B}^{\ast}+(\theta_{2}\wedge)^{\ast}+(\bar\theta_{1}\wedge)^{\ast} \quad \text{ and } \quad \bar\partial^{\ast}_{B,\theta_{1},\theta_{2}}=\bar\partial^{\ast}_{B}-(\bar\theta_{2}\wedge)^{\ast} +(\theta_{1}\wedge)^{\ast}.
\]
Since we have $(\theta_{i}\wedge)^{\ast}=\ast_{T}(\bar\theta_{i}\wedge)\ast_{T}$,
by the Lemma \ref{id1} and \ref{id2},
we obtain the following K\"ahler identity (cf. \cite[Section 1.2]{AK}, \cite[Page 15]{Sim}).
\begin{lemma}\label{id3}
\[\Lambda\partial_{B,\theta_{1},\theta_{2}}-\partial_{B,\theta_{1},\theta_{2}}\Lambda=\sqrt{-1}\bar\partial^{\ast}_{B,\theta_{1},\theta_{2}}\quad \text{ and } \quad \Lambda\bar\partial_{B,\theta_{1},\theta_{2}}-\bar\partial_{B,\theta_{1},\theta_{2}}\Lambda=-\sqrt{-1}\partial^{\ast}_{B,\theta_{1},\theta_{2}}.
\]
\end{lemma}
Define 
\[\Delta^{\prime}_{B,\theta_{1},\theta_{2}}=\partial_{B,\theta_{1},\theta_{2}}\partial^{\ast}_{B,\theta_{1},\theta_{2}}+\partial^{\ast}_{B,\theta_{1},\theta_{2}}\partial_{B,\theta_{1},\theta_{2}}\quad \text{ and } \quad\Delta^{\prime\prime}_{B,\theta_{1},\theta_{2}}=\bar\partial_{B,\theta_{1},\theta_{2}}\bar\partial^{\ast}_{B,\theta_{1},\theta_{2}}+\bar\partial^{\ast}_{B,\theta_{1},\theta_{2}}\bar\partial_{B,\theta_{1},\theta_{2}}.\]
Then Lemma \ref{id3} implies $\Delta_{B,\phi}=2\Delta^{\prime\prime}_{B,\theta_{1},\theta_{2}}=2\Delta^{\prime}_{B,\theta_{1},\theta_{2}}$.
Denote ${\mathcal H}^{r}_{B}(M,\theta_{1},\theta_{2})=\Delta^{\prime\prime}_{B,\theta_{1},\theta_{2}\vert A_{B}^{r}(M)\otimes \C}$.
We have the Hodge decomposition
\[A_{B}^{r}(M)\otimes \C={\mathcal H}^{r}_{B}(M,\theta_{1},\theta_{2})\oplus {\rm im}\,\bar\partial_{B,\theta_{1},\theta_{2}\vert A_{B}^{r-1}(M)\otimes \C}\oplus {\rm im}\,\bar\partial^{\ast}_{B,\theta_{1},\theta_{2}\vert A_{B}^{r+1}(M)\otimes \C}.
\]
Hence we obtain  an isomorphism
\[ H^{r}_{B}(M,\phi)\cong H^{r}_{B}(M,\theta_{1},\theta_{2}).
\]
By Lemma \ref{tw1} and \ref{RRA}, we have the following result.
\begin{theorem}
For any $t\in \R^{\ast}$, we have the equation
\[\dim  H^{1}(M,\phi)=\dim H^{1}(M,t\theta_{1}+t\bar\theta_{1}+\theta_{2}-\bar\theta_{2}).\]

\end{theorem}
Since we have isomorphisms $H^{1}(M,\C)\cong H^{1}_{B}(M,\C)\cong {\mathcal H}^{1}_{B}(M)\otimes\C$, 
we obtain Theorem \ref{locii}.

\section{Proof of Corollary \ref{VNIL}}
Let $G$ be a simply connected solvable Lie group and $N$ be the nilradical (i.e. maximal connected nilpotent normal subgroup) of $G$.
Denote by $\g$ the Lie algebra of $G$.
Then we can take a simply connected nilpotent subgroup $C \subset G$ such that $G = C \cdot N$ (see \cite[Proposition 3.3]{dek}).
Since $C$ is nilpotent, the map 
\[\Phi: C\ni c\mapsto ({\rm Ad}_{c})_{s}\in {\rm Aut}(\g)
\]
is a diagonalizable representation
where $({\rm Ad}_{c})_{s}$ is the semi-simple part of the adjoint operator ${\rm Ad}_{c}$ (see \cite{KDD}). 
We take a diagonalization $\Phi={\rm diag}(\alpha_{1},\dots,\alpha_{n})$ where $\alpha_{1},\dots,\alpha_{n}$ are $\C$-valued characters of $C$.
Since the adjoint representation ${\rm Ad}$ on the nilradical $N$ is unipotent and we have an isomorphism $G/N\cong C/C\cap N$,  $\alpha_{1},\dots,\alpha_{n}$ are considered as characters of $G$.
For each $g\in G$, $\alpha_{i}(g)$ is an eigenvalue of ${\rm Ad}_{g}$.

Suppose $G$ admits a lattice $\Gamma$. 
We consider the solvmanifold $G/\Gamma$.
$G/\Gamma$ is an aspherical manifold with the fundamental group $\Gamma$.
In \cite{KHH}, the set ${\mathcal J}_{1}(\Gamma) $ was studied.
The author proved that ${\mathcal J}_{1}(\Gamma) $ is a finite set (\cite[Corollary 5.9]{KHH}) and if one of the characters $\alpha_{1},\dots,\alpha_{n}$ is non-unitary, then there exists a non-unitary character $\rho\in {\mathcal C}(\Gamma) $ of $\Gamma$ such that $\rho\in {\mathcal J}_{1}(\Gamma) $ (\cite[Corollary 5.10]{KHH}).
We suppose that $\Gamma$ can be the fundamental group of a compact Sasakian manifold.
Then by Corollary \ref{FIIIn}, characters $\alpha_{1},\dots,\alpha_{n}$ are all unitary.
Hence, for any $g\in G$, all eigenvalues of ${\rm Ad}_{g}$ are unitary.
In this case, a lattice $\Gamma$ of $G$ is virtually nilpotent (see \cite[Chapter IV. 5]{Aus}).
It is known that every polycyclic group contains a lattice of some simply connected solvable Lie group as a finite index normal subgroup (see \cite[Theorem 4.28]{R}). 
Hence Corollary \ref{VNIL} follows.

\end{document}